\documentclass{article}
\usepackage{amsmath,amsfonts,amsthm,amssymb,amscd,color,xcolor,mathrsfs}
 \usepackage{hyperref}

\colorlet{darkblue}{blue!50!black}

\hypersetup{
    colorlinks,%
    citecolor=darkblue,%
    filecolor=red,%
    linkcolor=darkblue,%
    urlcolor=magenta,%
    pdfnewwindow=true,%
    pdfstartview={FitH}
}

\newcommand{\esssup}{\mathop{\rm ess\ sup}}

\newcommand{\p}{\partial}
\newcommand{\e}{\varepsilon}

\newcommand{\R}{{\mathbb R}}

\newcommand{\Z}{{\mathbb Z}}

\newcommand{\BB}{{\cal B}}

\newcommand{\FF}{{\cal F}}

\newcommand{\KK}{{\cal K}}

\theoremstyle{plain}

\newtheorem{theorem}{Theorem}[section]
\newtheorem{lemma}[theorem]{Lemma}
\newtheorem{proposition}[theorem]{Proposition}

\theoremstyle{definition}
\newtheorem{definition}[theorem]{Definition}

\theoremstyle{remark}

\newtheorem{remark}[theorem]{Remark}

\begin{document}
\author{Armen Shirikyan\footnote{Department of Mathematics, University of Cergy--Pontoise, CNRS UMR 8088, 2 avenue Adolphe Chauvin, 95302 Cergy--Pontoise, France; e-mail: \href{mailto:Armen.Shirikyan@u-cergy.fr}{Armen.Shirikyan@u-cergy.fr}}}
\title{Approximate controllability of the viscous Burgers equation on the real line}
\date{}
\maketitle

\begin{abstract}
The paper is devoted to studying the 1D viscous Burgers equation controlled by an external force. It is assumed that the initial state is essentially bounded, with no decay condition at infinity, and the control is a trigonometric polynomial of low degree with respect to the space variable. We construct explicitly a control space of dimension~$11$ that enables one to steer the system to any neighbourhood  of a given final state in local topologies. The proof of this result is based on an adaptation of the Agrachev--Sarychev approach to the case of an unbounded domain.
 

\end{abstract}

\section{Introduction}
\label{s1}
Let us consider the following viscous Burgers equation on the real line:
\begin{equation} \label{1.1}
\p_t u-\mu\p_x^2u+u\p_xu=f(t,x), \quad x\in\R. 
\end{equation}
Here $u=u(t,x)$ is an unknown function, $\mu>0$ is a viscosity coefficient, and~$f(t,x)$ is an external force which is assumed to be essentially bounded in~$x$ and integrable in~$t$. Equation~\eqref{1.1} is supplemented with the initial condition
\begin{equation} \label{1.2}
u(0,x)=u_0(x),
\end{equation}
where $u_0\in L^\infty(\R)$. Due to the maximum principle, one can easily prove the existence and uniqueness of a solution for~\eqref{1.1}, \eqref{1.2} in appropriate functional classes. Our aim is to study controllability properties of~\eqref{1.1}. Namely, we assume that~$f$ has the form
\begin{equation} \label{1.3}
f(t,x)=h(t,x)+\eta(t,x),
\end{equation}
where $h$ is a fixed regular function and~$\eta$ is a control, which is assumed to be a smooth function in time with range in a finite-dimensional subspace $E\subset L^\infty(\R)$. We shall say that~\eqref{1.1} is {\it approximately controllable\/} at a time~$T>0$ if for any initial state $u_0\in L^\infty(\R)$, any target $\hat u\in C(\R)$,  and any numbers $\e,r>0$ there is a smooth function $\eta:[0,T]\to E$ such that the solution~$u(t,x)$ of problem~\eqref{1.1}--\eqref{1.3} satisfies the inequalities
\begin{equation} \label{1.4}
\bigl\|u(T,\cdot)\bigr\|_{L^\infty(\R)}\le K, \quad
\bigl\|u(T,\cdot)-\hat u\bigr\|_{L^\infty([-r,r])}<\e,
\end{equation}
where $K>0$ does not depend on~$r$ and~$\e$. 
Given a finite subset $\Lambda\subset\R$, we denote by $E_\Lambda$ the vector space spanned by the functions~$\cos(\lambda x)$ and~$\sin(\lambda x)$ with $\lambda\in\Lambda$. The following theorem is a weaker version of the main result of this paper. 

\medskip
\noindent
{\bf  Main Theorem.}
{\it Let $\Lambda=\{0,\lambda_1,\lambda_2,2\lambda_1,2\lambda_2,\lambda_1+\lambda_2\}$, where~$\lambda_1$ and~$\lambda_2$ are incommensurable positive numbers, and let $E=E_\Lambda$. Then Eq.~\eqref{1.1} is approximately controllable at any time~$T>0$.}

\medskip
We refer the reader to Section~\ref{s2} for a stronger result on approximate controllability and for an outline of its proof, which is based on an adaptation of a general approach introduced by Agrachev and Sarychev in~\cite{AS-2005} and further developed in~\cite{AS-2008}; see also~\cite{shirikyan-cmp2006,nersisyan-2011,sarychev-2012} for some other extensions. Let us note that the Agrachev--Sarychev approach enables one to establish a much stronger property: given any initial and target states and any non-degenerate finite-dimensional functional, one can construct a control that steers the system to the given neighbourhood of the target  so that the values of the functional on the solution and on the target coincide. However, to make the presentation simpler and shorter, we confine ourselves to the approximate controllability. The above-mentioned property of controllability will be analysed in~\cite{shirikyan-2014} in the more difficult case of the 2D Navier--Stokes system.  

The main theorem stated above proves the approximate controllability of the Burgers equation by a control whose Fourier transform is localised at~$11$ points. This result is in sharp contrast with the case of a control localised in the physical space, for which the approximate controllability does not hold even for the problem in a bounded interval. This fact was established by Fursikov and Imanuvilov; see Section~I.6 of the book~\cite{FI1996}. Other negative results on controllability of the Burgers equation via boundary were obtained by Diaz~\cite{diaz-1996} and Guerrero and Imanuvilov~\cite{GI-2007}. On the other hand, Coron showed in~\cite{coron-2007} that any initial state can be driven to zero by a boundary control and Fern\'andez-Cara and Guerrero~\cite{FG-2007} proved the exact controllability (with an estimate for the minimal time of control) for the problem with distributed control. Furthermore, Glass and Guerrero~\cite{GG-2007} established global controllability to non-zero constant states via boundary for small values of the viscosity and Chapouly~\cite{chapouly-2009} proved the global exact controllability to a given solution by two boundary and one distributed controls. Imanuvilov and Puel~\cite{IP-2009}  proved the global boundary controllability of the 2D Burgers equation in a bounded domain under some geometric conditions. We refer the reader to the book~\cite{coron2007} for a discussion of the methods used in the control theory for the Burgers equation on a bounded interval. To the best of our knowledge, the problem of controllability of the viscous Burgers equation was not studied in the case of an unbounded domain. 

\smallskip
The paper is organised as follows. In Section~\ref{s2}, we formulate the main result and outline the scheme of its proof. Section~\ref{s3} collects some facts about the Cauchy problem for Eq.~\eqref{1.1} without decay condition at infinity. The proof of the main result of the paper is given in Section~\ref{s4}. 

\subsection*{Acknowledgments} 
This research was carried out within the MME-DII Center of Excellence (ANR-11-LABX-0023-01) and supported by the ANR grant STOSYMAP (ANR 2011 BS01 015 01).

\subsection*{Notation}  
Let $J\subset\R$ be a bounded closed interval, let $D\subset\R^n$ be an open subset, and let~$X$ be a Banach space. We denote by~$B_X(R)$ the closed ball in~$X$ of radius~$R$ centred at zero. We shall use the following functional spaces.

\medskip
\noindent
For $p\in[1,\infty]$, we denote by $L^p(J,X)$ the space of measurable functions $f:J\to X$ such that
$$
\|f\|_{L^p(J,X)}:=\biggl(\int_J\|f(t)\|_X^p\biggr)^{1/p}<\infty.
$$
In the case $p=\infty$, this norm should be replaced by $\esssup_{t\in J}\|f(t)\|_X$. 

\smallskip
\noindent
For an integer $k\in[0,+\infty]$, we write $C^k(J,X)$ for the space of~$k$ times continuously differentiable functions on~$J$ with range in~$X$ and endow it with  natural norm. In the case $k=0$, we omit the corresponding superscript. 

\smallskip
\noindent
For an integer $s\ge0$, we denote by~$H^s(D)$ the Sobolev space on~$D$ of order~$s$ with the standard norm~$\|\cdot\|_s$. In the case $s=0$, we write~$L^2(D)$ and~$\|\cdot\|$. 

\smallskip
\noindent
$L^\infty=L^\infty(\R)$ is the space of bounded measurable functions $f:\R\to\R$ with the natural norm~$\|f\|_{L^\infty}$. The space $L^\infty(D)$ is defined in a similar way.

\smallskip
\noindent
$W^{k,\infty}(\R)$ is the space of  functions $f\in L^\infty$ such that $\p_x^jf\in L^\infty$ for $0\le j\le k$. 

\smallskip
\noindent
$C_b^\infty=C_b^\infty(\R)$ stands for the space of infinitely differentiable functions $f:\R\to\R$ that are bounded together with all their derivatives. 

\smallskip
\noindent
$H_{\rm ul}^s=H_{\rm ul}^s(\R)$ is the space of functions $f:\R\to\R$ whose restriction to any bounded interval~$I\subset\R$ belongs~$H^s(I)$ such that
$$
\|f\|_{H_{\rm ul}^s}:=\sup_{x\in\R}\|f(x+\cdot)\|_{H^s([0,1])}<\infty. 
$$

\noindent
If $J=[a,b]$ and $X=H_{\rm ul}^s\mbox{ or }H_{\rm ul}^s\cap L^\infty$, then $C_*(J,X)$ stands for the space of functions $f:J\to X$ that are bounded and continuous on the interval~$(a,b]$ and possess a limit in the space~$H_{\rm loc}^s$ as $t\to a^+$. 

\smallskip
\noindent
We denote by~$C_i$ unessential positive constants. 

\section{Main result and scheme of its proof}
\label{s2}

We begin with the definition of the property of approximate controllability. As it will be proved in Section~\ref{s3}, the Cauchy problem~\eqref{1.1}, \eqref{1.2} is well posed. In particular, for any $T>0$, any integer $s\ge0$, and any functions $u_0\in L^\infty(\R)$ and $f\in L^1(J_T,H_{\rm ul}^s\cap L^\infty)$, there is a unique solution $u\in C_*(J_T,H_{\rm ul}^s\cap L^\infty)$ for~\eqref{1.1}, \eqref{1.2}. 

\begin{definition}
Let $T>0$, let $h\in L^1(J_T,H_{\rm ul}^s)$ for any $s\ge0$, and let $E\subset C_b^\infty$ be a finite-dimensional subspace. We shall say that problem~\eqref{1.1}, \eqref{1.3} is {\it approximately controllable at time~$T$ by an $E$-valued control\/} if for any integer $s\ge0$, any numbers $\e, r>0$, and any functions $u_0\in L^\infty$ and $\hat u\in H_{\rm ul}^s$ there is $\eta\in C^\infty(J_T,E)$ such that the solution $u(t,x)$ of~\eqref{1.1}--\eqref{1.3} satisfies the inequalities\footnote{Recall that the norm on the intersection of two Banach spaces is defined as the sum of the norms.}
\begin{equation} \label{2.1}
\|u(T,\cdot)\|_{H_{\rm ul}^s\cap L^\infty}\le K_s, \quad 
\|u(T,\cdot)-\hat u\|_{H^s([-r,r])}<\e,
\end{equation}
where $K_s>0$ is a constant depending only on~$\|u_0\|_{L^\infty}$, $\|\hat u\|_{H_{\rm ul}^s}$, $T$, and~$s$ (but not on~$r$ and~$\e$). 
\end{definition}

Recall that, given a finite subset $\Lambda\subset\R$, we denote by $E_\Lambda\subset C_b^\infty$ the vector span of the functions~$\cos(\lambda x)$ and~$\sin(\lambda x)$ with $\lambda\in\Lambda$. The following theorem is the main result of this paper. 

\begin{theorem} \label{t2.1}
Let $T>0$, $h\in L^2(J_T,H_{\rm ul}^s)$ for any $s\ge0$, let~$\lambda_1$ and~$\lambda_2$ be incommensurable positive numbers, and let 
$\Lambda=\{0,\lambda_1,\lambda_2,2\lambda_1,2\lambda_2,\lambda_1+\lambda_2\}$. Then problem~\eqref{1.1}, \eqref{1.3} is approximately controllable at time~$T$ by an $E_\Lambda$-valued control. 
\end{theorem}

A proof of this theorem is given in Section~\ref{s4}. Here we outline its scheme. Let us fix an integer $s\ge0$ and functions $u_0\in L^\infty$ and $\hat u\in H_{\rm loc}^s$. In view of the regularising property of the resolving operator for~\eqref{1.1} (see Proposition~\ref{p3.3}), there is no loss of generality in assuming that $u_0\in C_b^\infty$, and by a density argument, we can also assume that $\hat u\in C_b^\infty$. Furthermore, as it is proved in Section~\ref{s4.5}, if inequalities~\eqref{2.1} are established for $s=0$, then simple interpolation and regularisation arguments show that it remains true for any $s\ge1$. Thus, it suffices to prove~\eqref{2.1} for $s=0$. 

\smallskip
Given a finite-dimensional subspace $G\subset C_b^\infty$, we consider the controlled equations
\begin{align}
\p_t u-\mu\p_x^2u+\BB(u)&=h(t,x)+\eta(t,x),\label{2.2}\\
\p_t u-\mu\p_x^2(u+\zeta(t,x))+\BB(u+\zeta(t,x))&=h(t,x)+\eta(t,x),\label{2.3}
\end{align}
where $\eta$ and $\zeta$ are $G$-valued controls, and we set $\BB(u)=u\p_xu$. We say that Eq.~\eqref{2.2} is {\it $(\e,r,G)$-controllable at time~$T$ for the pair $(u_0,\hat u)$\/} (or simply {\it $G$-controllable\/} if the other parameters are fixed) if one can find $\eta\in C^\infty(J_T,G)$ such that the solution~$u$ of~\eqref{2.2}, \eqref{1.2} satisfies inequalities~\eqref{2.1} with $s=0$. The concept of $(\e,r,G)$-controllability for~\eqref{2.3} is defined in a similar way. 

We need to prove that~\eqref{2.2} is {\it $E_\Lambda$-controllable\/}. This fact will be proved in four steps. From now on, we assume that functions $u_0,\hat u\in C_b^\infty(\R)$ and the positive numbers~$T$, $\e$, and~$r$ are fixed and do not indicate explicitely the dependence of other quantities on them. 

\medskip
\underline{\sl Step~1: Extension}. 
Let us fix a finite-dimensional subspace $G\subset C_b^\infty$. 
Even though Eq.~\eqref{2.3} contains more control functions than Eq.~\eqref{2.2}, the property of  $G$-controllability is equivalent for them. Namely, we have the following result.

\begin{proposition} \label{p2.3}
Equation~\eqref{2.2} is $G$-controllable if and only if so is Eq.~\eqref{2.3}. 
\end{proposition}

\underline{\sl Step~2: Convexification}. 
Let us fix a subset $N\subset C_b^\infty$ invariant under multiplication by real numbers such that
\begin{equation} \label{2.5}
 N\subset G, \quad \BB(N)\subset G.
\end{equation}
We denote by $\FF(N,G)\subset C_b^\infty$ the  vector span of functions of the form
\begin{equation} \label{2.6}
\eta+\xi\p_x\tilde\xi+\tilde\xi\p_x\xi,
\end{equation}
where $\eta,\xi\in G$ and $\tilde\xi\in N$. It is easy to see that~$\FF(N,G)$ is a finite-dimensional subspace contained in the convex envelope of~
$G$ and~$\BB(G)$; cf.\ Lemma~\ref{l5.1} in Section~\ref{s4.2}. The following proposition  is an infinite-dimensional analogue of the well-known convexification principle for controlled ODE's (e.g., see~\cite[Theorem~8.7]{AS2004}). 

\begin{proposition} \label{p2.5}
Under the above hypotheses, Eq.~\eqref{2.3} is $G$-controllable if and only if Eq.~\eqref{2.2} is $\FF(N,G)$-controllable.
\end{proposition}

\underline{\sl Step~3: Saturation}. 
Propositions~\ref{p2.3} and~\ref{p2.5} (and their proof) imply the following result, which is a kind of ``relaxation property'' for  the controlled Burgers equation. 

\begin{proposition} \label{p2.6}
Let $N,G\subset C_b^\infty$ be as in Step~2. Then Eq.~\eqref{2.2} is $G$-controllable if and only if it is $\FF(N,G)$-controllable. Moreover, the constant~$K_0$ of~\eqref{2.1} corresponding to Eq.~\eqref{2.2} with $G$-valued control can be made arbitrarily close to that for Eq.~\eqref{2.2} with $\FF(N,G)$-valued control. 
\end{proposition}

We now set $N=\{c\cos(\lambda_1 x), c\sin(\lambda_1 x), c\cos(\lambda_2 x), c\sin(\lambda_2 x), c\in\R\}$ and define $E_k=\FF(N,E_{k-1})$ for $k\ge1$, where $E_0=E_\Lambda$. Note that $\BB(N)\subset E_\Lambda$  (this inclusion will be important in the proof of Lemma~\ref{l5.1}). It follows from Proposition~\ref{p2.5}  that Eq.~\eqref{2.2} is $E_\Lambda$-controllable if and only if it is $E_k$-controllable for some integer~$k\ge1$. We shall show that the latter property is true for a sufficiently large~$k$. To this end, we first establish the following saturation property: there is a dense countable subset $\Lambda_\infty\subset\R_+$ such that
\begin{equation} \label{2.7}
\bigcup_{k=1}^\infty E_k\mbox{ contains the functions $\sin(\lambda x)$ and $\cos(\lambda x)$ with $\lambda\in\Lambda_\infty$}.
\end{equation}

\underline{\sl Step~4: Large control space}. 
Once~\eqref{2.7} is proved, one can easily show that~\eqref{2.2} is $E_k$-controllable for a sufficiently large~$k$. To this end, it suffices to join~$u_0$ and~$\hat u$ by a smooth curve, to use Eq.~\eqref{2.2} to define the corresponding control~$\eta$, and to approximate it, in local topologies,  by functions belonging to~$E_k$. The fact that the corresponding solutions are close follows from continuity of the resolving operator for~\eqref{2.2} in local norms (see Proposition~\ref{p3.4}). This will complete the proof of Theorem~\ref{t2.1}.

\section{Cauchy problem}
\label{s3}
In this section, we discuss the existence and uniqueness of a solution for the Cauchy problem for the generalised Burgers equation 
\begin{equation} \label{3.1}
\p_t u-\mu\p_x^2\bigl(u+g(t,x)\bigr)+\BB\bigl(u+g(t,x)\bigr)=f(t,x), \quad x\in\R, 
\end{equation}
where $f$ and $g$ are given functions. We also establish some a priori estimates for higher Sobolev norms and Lipschitz continuity of the resolving operator in local norms. The techniques of the maximum principle and of weighted energy estimates enabling  one to derive this type of results are well known, and sometimes we confine ourselves to the formulation of a result and a sketch of its proof. 

\subsection{Existence, uniqueness, and regularity of a solution}
\label{s3.1}
Before studying the well-posedness of the Cauchy problem for Eq.~\eqref{3.1}, we recall some results for the linear equation
\begin{equation} \label{3.2}
\p_t v-\mu\p_x^2 v+a(t,x)\p_xv+b(t,x)v=c(t,x), \quad x\in\R,
\end{equation}
supplement with the initial condition
\begin{equation} \label{3.3}
v(0,x)=v_0(x),
\end{equation}
where $v_0\in L^\infty(\R)$. The following proposition establishes the existence, uniqueness, and a priori estimates for a solution of problem~\eqref{3.1}, \eqref{3.2} in spaces with no decay condition at infinity.

\begin{proposition} \label{p3.1}
Let $T>0$ and let $a$, $b$, $c$, and $f$ be some functions such that
$$
a\in L^2(J_T,L^\infty), \quad b,c\in L^1(J_T,L^\infty),
$$
Then for any $v_0\in L^\infty$ problem~\eqref{3.2}, \eqref{3.3} has a unique solution $v(t,x)$ such that
$$
v\in L^\infty(J_T\times\R)\cap C_*(J_T,L_{\rm ul}^2), \quad \|\p_xv(\cdot,x)\|_{L^2(J_T)}\in L_{\rm ul}^2. 
$$
Moreover, this solution satisfies the inequalities
\begin{align}
\|v\|_{L^\infty(J_t\times\R)}&\le \exp\bigl(\|b\|_{L^1(J_t,L^\infty)}\bigr)\Bigl(\|v_0\|_{L^\infty}
+\|c\|_{L^1(J_t,L^\infty)}\Bigr),\label{3.4}\\
\|v(t)\|_{L_{\rm ul}^2}+\|\p_xv\|_{L_{\rm ul}^2L^2(J_t)}
&\le C\,e^{C(\bar a(t)+\bar b(t))}
\Bigl(\|v_0\|_{L_{\rm ul}^2}+\|c\|_{L_{\rm ul}^2L^2(J_t)}\Bigr),\label{3.5}
\end{align}
where $0\le t\le T$, $C>0$ is an absolute constant, and
$$
{\bar b}(t)=\|b\|_{L^2(J_t,L_{\rm ul}^2)}^2, \quad 
{\bar a}(t)=\|a\|_{L^2(J_t,L^\infty)}^2, \quad 
\|c\|_{L_{\rm ul}^2L^2(J_t)}=\sup_{y\in\R}\|c\|_{L^2(J_t\times[y,y+1])}. 
$$
If, in addition, we have $a\in L^\infty(J_T\times\R)$, then $u\in L^p(J_T,H_{\rm ul}^1)$ for any $p\in[1,\frac43)$ and 
\begin{equation} \label{3.6}
\|v\|_{L^p(J_t,H_{\rm ul}^1)}
\le C_1\Bigl(\|v_0\|_{L_{\rm ul}^2}+\int_0^t\|c(r)\|_{L_{\rm ul}^2}dr\Bigr),
\end{equation}
where $C_1>0$ depends only on~$p$, $\|a\|_{L^\infty}$, and $\|b\|_{L^2(J_T,L_{\rm ul}^2)}$.
\end{proposition}

\begin{proof}
Inequality~\eqref{3.4} is nothing else but the maximum principle, while~\eqref{3.5} can easily be obtained on multiplying~\eqref{3.2} by $e^{-|x-y|}v$, integrating over~$x\in\R$, and taking the supremum over~$y\in\R$. Once these a priori estimates are established (by a formal computation), the existence and uniqueness of a solution in the required functional classes can be proved by usual arguments (e.g., see~\cite{lr2002} for the more complicated case of the Navier--Stokes equations), and we omit them. The only non-standard point is inequality~\eqref{3.6}, and we now briefly outline its proof. 

Let $K_t(x)$ be the heat kernel on the real line:
\begin{equation} \label{3.00}
K_t(x)=\frac{1}{\sqrt{4\pi\mu t}}\,\exp\bigl(-\tfrac{x^2}{4\mu t}\bigr), 
\quad x\in\R, \quad t>0.
\end{equation}
The following estimates are easy to check:
\begin{equation} \label{3.7}
\|K_t*g\|_{L_{\rm ul}^2}\le \|g\|_{L_{\rm ul}^2}, \quad 
\|\p_x(K_t*g)\|_{L_{\rm ul}^2}\le C_1t^{-\frac34}\|g\|_{L_{\rm ul}^2}, \quad t>0.
\end{equation}
Here and henceforth, the constants~$C_i$ in various inequalities may depend on~$\mu$ and~$T$. We now use the Duhamel formula to write a solution of~\eqref{3.2}, \eqref{3.3} in the form
$$
v(t,x)=(K_t*v_0)(x)+\int_0^tK_{t-r}*\bigl(c(r)-a\p_xv(r)-b v(r)\bigr)\,dr. 
$$
It follows from~\eqref{3.7} that
\begin{align*}
\|v(t)\|_{H_{\rm ul}^1}
&\le C_1t^{-\frac34}\|v_0\|_{L_{\rm ul}^2}
+C_2\int_0^t(t-r)^{-\frac34}\bigl(\|c\|_{L_{\rm ul}^2}
+\|a\|_{L^\infty}\|v\|_{H_{\rm ul}^1}+\|b\|_{L_{\rm ul}^2}\|v\|_{L^\infty}\bigr)\,dr\\ 
&\le C_1t^{-\frac34}\|v_0\|_{L_{\rm ul}^2}
+C_2\int_0^t(t-r)^{-\frac34}\bigl(\|c\|_{L_{\rm ul}^2}
+\bigl(\|a\|_{L^\infty}+1\bigr)\,\|v\|_{H_{\rm ul}^1}\bigr)dr\\
&\qquad+C_3\int_0^t(t-r)^{-\frac34}\|b\|_{L_{\rm ul}^2}^2\|v\|_{L_{\rm ul}^2}\,dr\,,
\end{align*}
where we used the interpolation inequality 
$\|v\|_{L^\infty}^2\le C\|v\|_{L_{\rm ul}^2}\|v\|_{H_{\rm ul}^1}$. 
Taking the left- and right-hand sides of this inequality to the~$p^{\text{th}}$ power, integrating in time, and using~\eqref{3.5}, after some simple transformations  we obtain the following differential inequality for the increasing function $\varphi(t)=\int_0^t\|v(r)\|_{H_{\rm ul}^1}^pdr$:
$$
\varphi(t)\le C_4 Q^p+C_4\Bigl(\int_0^t\|c(r)\|_{L_{\rm ul}^2}dr\Bigr)^p
+C_4\bigl(\|a\|_{L^\infty(J_t\times\R)}^p+1\bigr)\int_0^t(t-r)^{-\frac34}\varphi(r)\,dr,
$$
where $Q$ stands for the expression in the brackets on the right-hand side of~\eqref{3.6}, and~$C_4$ depends on $\bar a(T)$, $\bar b(T)$, $T$, and~$\mu$. A Gronwall-type argument enables one to derive~\eqref{3.6}. 
\end{proof}

Let us note that inequality~\eqref{3.5} does not use the fact that $b,c\in L^1(J_T,L^\infty)$ and remains valid for any coefficient $b\in L^2(J_T,L_{\rm ul}^2)$ and any right-hand side~$c$ for which $\|c\|_{L_{\rm ul}^2L^2(J_T)}<\infty$. This observation will be important in the proof of Theorem~\ref{t3.2}. 

\smallskip
We now turn to the Burgers equation~\eqref{3.1}, supplemented with the initial condition~\eqref{1.2}. The proof of the following result is carried out by standard arguments, and we only sketch the main ideas.

\begin{theorem} \label{t3.2}
Let  $f\in L^1(J_T,L^\infty)$ and $g\in L^\infty(J_T\times\R)\cap L^2(J_T,W^{1,\infty})\cap L^1(J_T,W^{2,\infty})$ for some $T>0$ and let $u_0\in L^\infty$. Then problem~\eqref{3.1}, \eqref{1.2} has a unique solution $u(t,x)$ such that
\begin{equation} \label{3.9}
u\in L^\infty(J_T\times\R)\cap C_*(J_T,L_{\rm ul}^2)\cap L^p(J_T,H_{\rm ul}^1), 
\quad \|\p_xu(\cdot,x)\|_{L^2(J_T)}\in L_{\rm ul}^2,
\end{equation}
where $p\in[1,\frac43)$ is arbitrary. 
Moreover, the mapping $(u_0,f,g)\mapsto u$ is uniformly Lipschitz continuous (in appropriate spaces) on every ball. 
\end{theorem}

\begin{proof}
To prove the existence, we first derive some a priori estimates for a solution, assuming that it exists. Let us assume that the functions~$u_0$, $f$, and~$g$ belong to the balls of radius~$R$ centred at zero in the corresponding spaces. If a function~$u$ satisfies~\eqref{3.1}, then it is a solution of the linear equation~\eqref{3.2} with 
$$
a=u+g, \quad b=\p_xg, \quad c=f+\mu\p_x^2g-g\p_xg. 
$$
It follows from~\eqref{3.4} that 
\begin{equation} \label{3.10}
\|u\|_{L^\infty(J_T\times\R)}\le C_1(R). 
\end{equation}
Inequalities~\eqref{3.5} and~\eqref{3.6} now imply that
\begin{equation} \label{3.8}
\|u\|_{L^\infty(J_T,{L_{\rm ul}^2})}+\|u\|_{L^p(J_T,H_{\rm ul}^1)}
+\|u\|_{H_{\rm ul}^1L^2(J_T)}\le C_2(R). 
\end{equation}
We have thus established some bounds for the norm of a solution in the spaces entering~\eqref{3.9}. The local existence of a solution can now be proved by a fixed point argument, whereas the absence of finite-time blowup follows from the above a priori estimates. 

Let us prove a Lipschitz property for the resolving operator, which will imply, in particular, the uniqueness of a solution. Assume that $u_i$, $i=1,2$, are two solutions corresponding to some data $(u_{0i},f_i,g_i)$ that belong to balls of radius~$R$ centred at zero in the corresponding spaces. Setting $v=u_1-u_2$, $f=f_1-f_2$, $g=g_1-g_2$, and~$v_0=u_{01}-u_{02}$, we see that~$v$ satisfies~\eqref{3.2}, \eqref{3.3} with 
$$
a=u_1+g_1, \quad b=\p_x(u_2+g_2), \quad 
c=f+\mu\p_x^2g-(u_1+g_1)\p_xg-g\p_x(u_2+g_2). 
$$
Multiplying Eq.~\eqref{3.2} by $e^{-|x-y|}v$, integrating in $x\in\R$, and using~\eqref{3.10} and~\eqref{3.8}, after some transformations we obtain
\begin{equation} \label{3.12}
\p_t\|v\|_y^2+\mu\|\p_xv\|_y^2\le C_3(R)\|v\|_y^2+2\|c\|_y\|v\|_y,
\end{equation}
where we set 
$$
\|w\|_y^2=\int_\R w^2(x)e^{-|x-y|}dx. 
$$
Application of a Gronwall-type argument implies that
\begin{equation} \label{3.13}
\|v(t)\|_y^2+\int_0^t \|\p_xv\|_y^2\,ds
\le C_4(R)\Bigl(\|v_0\|_y+\int_0^t\|c(s)\|_y\,ds\Bigr)^2. 
\end{equation}
Taking the square root and the supremum in $y\in\R$, we derive
\begin{equation} \label{3.14}
\|v\|_{L^\infty(J_t,L_{\rm ul}^2)}+\|\p_xv\|_{L_{\rm ul}^2L^2(J_t)}
\le C_5(R)\Bigl(\|v_0\|_{L_{\rm ul}^2}+\sup_{y\in\R}\int_0^t\|c(s)\|_y\,ds\Bigr).
\end{equation}
Now note that 
\begin{equation} \label{3.17}
\|c\|_y\le \|f\|_y+\mu\|\p_x^2g\|_y+\|u_1+g_1\|_{L^\infty}\|\p_xg\|_y+\|g\|_{L^\infty}\|\p_xu_2+\p_xg_2\|_y,
\end{equation}
whence it follows that
$$
\int_0^t\|c\|_y\,ds\le \|f\|_{L_{\rm ul}^2L^2(J_t)}
+C_6(R)\bigl(\|\p_x^2g\|_{L^1(J_t,L^\infty)}+\|\p_xg\|_{L^2(J_t,L^\infty)}+\|g\|_{L^2(J_t,L^\infty)}\bigr). 
$$
Substituting this inequality in~\eqref{3.14}, we obtain
\begin{equation} \label{3.15}
\|v\|_{L^\infty(J_t,L_{\rm ul}^2)}+\|\p_xv\|_{L_{\rm ul}^2L^2(J_t)}
\le C_8(R)\Bigl(\|v_0\|_{L_{\rm ul}^2}+\|f\|_{L_{\rm ul}^2L^2(J_t)}+|\!|\!|g|\!|\!|_t\Bigr),
\end{equation}
where we set 
$$
|\!|\!|g|\!|\!|_t=\|g\|_{L^1(J_t,W^{2,\infty})}+\|g\|_{L^2(J_t,W^{1,\infty})}.
$$
Inequality~\eqref{3.15} establishes the required Lipschitz property of the resolving operator. 
\end{proof}

\begin{remark} \label{r3.1}
An argument similar to that used in the proof of Theorem~\ref{t3.2}  enables one to estimate the $H_{\rm ul}^1$-norm of the difference between two solutions. Namely, let~$u_i(t,x)$, $i=1,2$, be two solutions of~\eqref{3.1}, \eqref{1.2} corresponding to some data
$$
(u_{0i},f_i,g_i)\in H_{\rm ul}^1\times L^2(J_T,L^\infty)\times L^\infty(J_T,W^{2,\infty}),
\quad i=1,2,
$$
whose norms do not exceed~$R$. Then the difference $v=u_1-u_2$ satisfies the inequality
\begin{equation} \label{3.33}
\|v\|_{L^\infty(J_T,H_{\rm ul}^1)}
\le C(R)\Bigl(\|v_0\|_{H_{\rm ul}^1}+\|f\|_{L^2(J_T,L_{\rm ul}^2)}
+\|g\|_{L^4(J_T,W^{2,\infty})}\Bigr),
\end{equation}
where we retained the notation used in the proof of~\eqref{3.15}. 
\end{remark}

Finally, the following proposition establishes a higher regularity of solutions for~\eqref{3.1} with $g\equiv0$, provided that the right-hand side is sufficiently regular. 

\begin{proposition} \label{p3.3}
Under the hypotheses of Theorem~\ref{t3.2}, assume that $f\in L^2(J_T,H_{\rm ul}^s)$ for an integer $s\ge1$ and $g\equiv0$. Then the solution~$u(t,x)$ constructed in Theorem~\ref{t3.2} belongs to~$C([\tau,T],H_{\rm ul}^s)$ for any $\tau>0$ and satisfies the inequality
\begin{equation} \label{3.16}
\sup_{t\in J_T}\bigl(t^{k}\|\p_x^ku(t)\|_{L_{\rm ul}^2}^2\bigr)+\sup_{y\in\R}\int_0^Tt^k\|\p_x^{k+1}u(t)\|_{L^2(I_y)}^2dt\le Q_k\bigl(\|u_0\|_{L^\infty}+\|f\|_{L^2(J_T,H_{\rm ul}^k\cap L^\infty)}\bigr), 
\end{equation}
where $0\le k\le s$, $I_y=[y,y+1]$, and $Q_{k}$ is an increasing function. Furthermore, if $u_0\in C_b^\infty$, then the solution belongs to $C(J_T,H_{\rm ul}^s)$, and inequality~\eqref{3.16} is valid without the factor of~$t^{k/2}$ on the left-hand side and~$\|u_0\|_{L^\infty}$ replaced by~$\|u_0\|_{H_{\rm ul}^k}$ on the right-hand side.
\end{proposition}

\begin{proof}
We confine ourselves to the derivation of the a priori estimate~\eqref{3.16} for $u_0\in L^\infty$. Once it is proved, the regularity of a solution can be obtained by standard arguments. Furthermore, the case when $u_0\in C_b^\infty$ can be treated by a similar, but simpler technique, and we omit it. 

\smallskip
The proof of~\eqref{3.16} is by induction on~$k$. For $k=0$, inequality~\eqref{3.16} is a consequence of~\eqref{3.8}. We now assume that $l\in[1,s]$ and that~\eqref{3.16} is established for all $k\le l-1$. Let us set
$$
\varphi_y(t)=t^l\int_\R e^{-\langle x-y\rangle}|\p_x^{l}u|^2dx=t^l\|\p_x^{l}u\|_y^2,
\quad y\in \R,
$$
where $\langle z\rangle=\sqrt{1+z^2}$. In view of~\eqref{3.1}, the derivative of~$\varphi_y$ can be written as
\begin{equation} \label{3.31}
\p_t\varphi_y(t)=l\,t^{l-1}\|\p_x^{l}u\|_y^2
+2t^l\int_\R e^{-\langle x-y\rangle}\p_x^lu\,\p_x^l(\p_x^2u-u\p_xu+f)\,dx. 
\end{equation}
Integrating by parts and using~\eqref{3.10} and the Cauchy--Schwarz inequality, we derive
\begin{align*}
\int_\R e^{-\langle x-y\rangle}\p_x^lu\,\p_x^{l+2}u\,dx
&\le -\|\p_x^{l+1}u\|_y^2+\|\p_x^{l+1}u\|_y\,\|\p_x^{l}u\|_y,\\
\int_\R e^{-\langle x-y\rangle}\p_x^lu\,\p_x^lf\,dx
&\le \|\p_x^{l}f\|_y\,\|\p_x^{l}u\|_y,\\
\int_\R e^{-\langle x-y\rangle}\p_x^lu\,\p_x^{l}(u\p_xu)\,dx
&\le \frac12\int_\R e^{-\langle x-y\rangle}\p_x^lu\,\p_x^{l+1}u^2\,dx\\
&\le \frac12\bigl(\|\p_x^{l+1}u\|_y+\|\p_x^{l}u\|_y\bigr)\|\p_x^lu^2\|_y. 
\end{align*}
Substituting these inequalities into~\eqref{3.31} and integrating in time, we obtain
$$
\varphi_y(t)+\int_0^t t^l\|\p_x^{l+1}u\|_y^2\,dt
\le \int_0^t\bigl(s^{l-1}\|\p_x^{l}u\|_y^2+4\varphi_y(s)+s^l\|\p_x^lu^2\|_y^2
+s^l\|\p_x^lf\|_y^2\bigr)\,ds.
$$
Taking the supremum over $y\in\R$ and using the induction hypothesis, we derive 
\begin{equation} \label{3.32}
\psi(t)\le Q_{l-1}+C_1\int_0^t\psi(s)\,ds +\sup_{y\in\R}\int_0^ts^l\|\p_x^lu^2\|_y^2ds+C_1\int_0^T\|f\|_{H_{\rm ul}^l}^2ds,
\end{equation}
where $Q_{l-1}$ is the function entering~\eqref{3.16} with $k=l-1$, and 
$$
\psi(t)=t^l\|\p_x^lu(t)\|_{L_{\rm ul}^2}^2+\sup_{y\in\R}\int_0^t t^l\|\p_x^{l+1}u\|_{L^2(I_y)}^2\,dt. 
$$
Now note that
$$
\int_0^ts^l\|\p_x^lu^2\|_y^2\,ds
\le C_2\|u\|_{L^\infty}^2\sum_{k\in\Z}e^{-|k-y|}\int_0^ts^l\|u\|_{H^l(I_k)}^2\,ds. 
$$
Substituting this into~\eqref{3.32} and using again the induction hypothesis and inequality~\eqref{3.10}, we obtain 
$$
\psi(t)\le C_3\int_0^t\psi(s)\,ds
+Q\bigl(\|u_0\|_{L^\infty}+\|f\|_{L^2(J_T,H_{\rm ul}^l\cap L^\infty)}\bigr),
$$
where $Q$ is an increasing function. Application of the Gronwall inequality completes the proof. 
\end{proof}

\subsection{Uniform continuity of the resolving operator in local norms}
\label{s3.3}

Theorem~\ref{t3.2} established, in particular, the Lipschitz continuity of the resolving operator for~\eqref{3.1}. The following proposition, which plays a crucial role in the next section, proves the uniform continuity of the resolving operator in local norms. 

\begin{proposition} \label{p3.4}
Under the hypotheses of Theorem~\ref{t3.2}, for any positive numbers~$T$, $R$, $r$, and~$\delta$ there are~$\rho$ and~$C$ such that, if triples $(u_{0i},f_i,g_i)$, $i=1,2$, satisfy the inclusions
$$
u_{0i}\in L^\infty,\quad f_i\in L^1(J_T,L^\infty), \quad g_i\in L^\infty(J_T\times\R)\cap L^2(J_T,W^{1,\infty})\cap L^1(J_T,W^{2,\infty}),
$$
and  corresponding norms are bounded by~$R$, then 
\begin{multline} \label{3.18}
\sup_{t\in J_T}\|u_1(t)-u_2(t)\|_{L^2([-r,r])}
\le \delta\\
+C\,\Bigl(\|u_{01}-u_{02}\|_{L^2(I_\rho)}+\|f_1-f_2\|_{L^1(J_T,L^2(I_\rho))}
+\|g_1-g_2\|_{L^2(J_T,H^2(I_\rho))}\Bigr),
\end{multline}
where $I_\rho=[-\rho,\rho]$, and $u_i(t)$ denotes the solution of~\eqref{3.1} issued from~$u_{0i}$. 
\end{proposition}

\begin{proof}
We shall use the notation introduced in the proof of Theorem~\ref{t3.2}. It follows from inequality~\eqref{3.13} with $y=0$ that 
\begin{equation} \label{3.19}
e^{-r/2}\|v(t)\|_{L^2(I_r)}\le C_1(R)\,
\Bigl(\|e^{-|\cdot|/2}v_0\|_{L^2}
+\int_0^T\|e^{-|\cdot|/2}c(t,\cdot)\|_{L^2}\,dt\Bigr).
\end{equation}
Now note that 
\begin{equation} \label{3.20}
\|e^{-|x|/2}v_0\|_{L^2}^2=\int_\R|v_0|^2e^{-|x|}dx\le \|v_0\|_{L^2(I_\rho)}^2+4e^{-\rho}\|v_0\|_{L_{\rm ul}^2}^2. 
\end{equation}
By a similar argument, we check that (cf.~\eqref{3.17})
\begin{multline*}
\|e^{-|\cdot|/2}c(t,\cdot)\|_{L^2}\le \|f\|_{L^2(I_\rho)}+\mu\|\p_x^2g\|_{L^2(I_\rho)}
+C_2(R)\,\|\p_xg\|_{L^2(I_\rho)}+C_3(R)e^{-\rho/2}\\
+\bigl(\|g\|_{L^\infty(I_\rho)}+e^{-\rho/4}\|g\|_{L^\infty}\bigr)\,
\|e^{-|\cdot|/4}(\p_xu_2+\p_xg_2)\|_{L^2(\R)}. 
\end{multline*}
Integrating in time and using~\eqref{3.8}, we obtain
\begin{equation} \label{3.21}
\int_0^T\|e^{-|\cdot|/2}c(t,\cdot)\|_{L^2}\,dt
\le C_4(R)\Bigl\{\int_0^T\|f\|_{L^2(I_\rho)}dt+\Bigl(\int_0^T\|g\|_{H^2(I_\rho)}^2dt\Bigr)^{1/2}+e^{-\rho/4}\Bigr\}.
\end{equation}
Substituting~\eqref{3.20} and~\eqref{3.21} into~\eqref{3.19} and taking~$\rho>0$ sufficiently large, we arrive at the required inequality~\eqref{3.18}.
\end{proof}

\section{Proof of Theorem~\ref{t2.1}}
\label{s4}

\subsection{Extension: proof of Proposition~\ref{p2.3}}
\label{s4.1}
We only need to prove that if Eq.~\eqref{2.3} is $G$-controllable, then so is~\eqref{2.2}, since the converse implication is obvious. Let $\tilde\eta,\tilde\zeta\in C^\infty(J_T,G)$ be such that the solution~$\tilde u$ of problem~\eqref{2.3}, \eqref{1.2} satisfies~\eqref{2.1} with $s=0$. In view of~\eqref{3.15}, replacing~$K_0$ by a slightly larger constant, we can assume that $\tilde\zeta(0)=\tilde\zeta(T)=0$. Let us set $u=\tilde u+\tilde\zeta$. Then~$u$ is a solution of~\eqref{2.2}, \eqref{1.2} with the control $\eta=\tilde\eta+\p_t\tilde\zeta$, which takes values in~$G$. Moreover, $u(T)=\tilde u(T)$ and, hence, $u$ satisfies~\eqref{2.1}. This completes the proof of Proposition~\ref{p2.3}, showing in addition that the constants~$K_0$ entering~\eqref{2.1} and corresponding to Eqs.~\eqref{2.2} and~\eqref{2.3} can be chosen arbitrarily close to each other. 

\subsection{Convexification: proof of Proposition~\ref{p2.5}}
\label{s4.2}
We begin with a number of simple observations. Let us set $G_1=\FF(N,G)$. By Proposition~\ref{p2.3}, if Eq.~\eqref{2.3} is $G$-controllability, then so is Eq.~\eqref{2.2}, and since $G\subset G_1$, we see that~\eqref{2.2} is $G_1$-controllable. Thus, it suffices to prove that if~\eqref{2.2} is $G_1$-controllable, then~\eqref{2.3} is $G$-controllable. To establish this property, it suffices to prove that, for any $\eta_1\in C^\infty(J_T,G_1)$ and any $\delta>0$ there are $\eta,\zeta\in L^\infty(J_T,G)$ such that the solution $u(t,x)$ of~\eqref{2.3}, \eqref{1.2} satisfies the inequality
\begin{equation} \label{4.1}
\|u(T)-u_1(T)\|_{H_{\rm ul}^1}<\delta, 
\end{equation}
where~$u_1$ stands for the solution of~\eqref{2.2}, \eqref{1.2} with $\eta=\eta_1$. Indeed, if this property is established, then we take two sequences $\{\eta^n\}, \{\zeta^n\}\subset C^\infty(J_T,G)$ such that (cf.~\eqref{3.33})
$$
\|\eta^n-\eta\|_{L^2(J_T,G)}+\|\zeta^n-\zeta\|_{L^4(J_T,G)}\to0
\quad\mbox{as $n\to\infty$}
$$
and denote by $u^n(t,x)$ the solution of~\eqref{2.3}, \eqref{1.2} with $\eta=\eta^n$ and $\zeta=\zeta^n$. It follows from~\eqref{3.33} that 
\begin{equation} \label{4.4}
\gamma_n:=\|u^n(T)-u(T)\|_{H_{\rm ul}^1}\to0\quad\mbox{as $n\to\infty$}.  
\end{equation}
Combining~\eqref{4.1} and~\eqref{4.4} and using the continuous embedding $H_{\rm ul}^1\subset L^\infty$, we derive
\begin{align*}
\|u^n(T)\|_{L^\infty}&\le \|u_1(T)\|_{L^\infty}+\|u(T)-u_1(T)\|_{L^\infty}+\|u^n(T)-u(T)\|_{L^\infty}\\
&\le K_0+C_1(\delta+\gamma_n),\\
\|u^n(T)-\hat u\|_{L^2(I_r)}
&\le \|u^n(T)-u(T)\|_{L^2(I_r)}+\|u(T)-u_1(T)\|_{L^2(I_r)}+\|u_1(T)-\hat u\|_{L^2(I_r)}\\
&\le C_2(\gamma_n+\delta)+\|u_1(T)-\hat u\|_{L^2(I_r)},
\end{align*}
where $I_r=[-r,r]$. Choosing~$\delta>0$ sufficiently small and~$n$ sufficiently large, we conclude that~$u^n$ satisfies inequalities~\eqref{2.1}, with a constant~$K_0$ arbitrarily close to that for~$u_1$. Finally, a similar approximation argument shows that, when proving~\eqref{4.1}, we can assume~$\eta_1(t)$ to be piecewise constant, with finitely many intervals of constancy. The construction of controls~$\eta, \zeta\in L^\infty(J_T,G)$ for which~\eqref{4.1} holds is carried out in several steps. 

\medskip
{\sl Step~1: An auxiliary lemma}. We shall need the following lemma, which establishes a relationship between $G$- and $\FF(N,G)$-valued controls. 

\begin{lemma} \label{l5.1}
For any $\eta_1\in\FF(N,G)$ and any $\nu>0$ there is an integer
$k\ge1$, numbers $\alpha_j>0$, and vectors $\eta,\zeta^j\in G$,
$j=1,\dots,k$, such that
\begin{align}
\sum_{j=1}^k\alpha_j&=1,\label{4.2}\\
\Bigl\|\eta_1-\BB(u)-\Bigl(\eta-\sum_{j=1}^k\alpha_j\bigl(\BB(u+\zeta^j)-\mu\p_x^2\zeta^j\bigr)\Bigr)\Bigr\|_{H_{\rm ul}^1}
&\le\nu\quad\mbox{for any $u\in H_{\rm ul}^1$}.\label{4.3}
\end{align}
\end{lemma}

\begin{proof}
It suffices to find functions $\eta,\tilde\zeta^j\in G$, $j=1,\dots,m$, such that
\begin{equation} \label{4.5}
\Bigl\|\eta_1-\eta+\sum_{j=1}^k\BB(\tilde\zeta^j)\Bigr\|_{H_{\rm ul}^1}\le\nu.
\end{equation}
Indeed, if such vectors are constructed, then we can set $k=2m$,
$$
\alpha_j=\alpha_{j+m}=\frac{1}{2m},\quad 
\zeta^j=-\zeta^{j+m}=\sqrt{m}\,\tilde\zeta^j\quad\mbox{for $j=1,\dots,m$},
$$
and relations~\eqref{4.2} and~\eqref{4.3} are easily checked. 

To construct $\eta,\tilde\zeta^j\in G$ satisfying~\eqref{4.5}, note that if $\eta_1\in\FF(N,G)$, then there are functions 
$\tilde\eta_j,\xi_j\in G$ and $\tilde\xi_j\in N$ such that
\begin{equation} \label{72}
\eta_1=\sum_{j=1}^k\bigl(\tilde\eta_j-\xi_j\p_x\tilde\xi_j-\tilde\xi_j\p_x\xi_j\bigr).
\end{equation}
Now note that, for any $\e>0$, 
$$
\xi_j\p_x\tilde\xi_j+\tilde\xi_j\p_x\xi_j=\BB(\e\xi_j+\e^{-1}\tilde\xi_j)-\e^2\BB(\xi_j)-\e^{-2}\BB(\tilde\xi_j).
$$
Combining this with~\eqref{72}, we obtain
$$
\eta_1-\sum_{j=1}^k\bigl(\tilde\eta_j+\e^{-2}\BB(\tilde\xi_j)\bigr)
+\sum_{j=1}^k\BB(\e\xi_j+\e^{-1}\tilde\xi_j)=\e^2\sum_{j=1}^k\BB(\xi_j).
$$
Choosing $\e>0$ sufficiently small and setting\footnote{Recall that $\BB(N)\subset G$, so that the vector~$\eta$ defined in~\eqref{100} belongs to~$G$.}
\begin{equation} \label{100}
\eta=\sum_{j=1}^k\bigl(\tilde\eta_j+\e^{-2}\BB(\tilde\xi_j)\bigr), \quad
\tilde\zeta^j=\e\xi_j+\e^{-1}\tilde\xi_j, 
\end{equation}
we arrive at the required inequality~\eqref{4.5}. 
\end{proof}

%\medskip
{\sl Step~2: Comparison with an auxiliary equation}.
Let $\eta_1\in L^\infty(J_T,G_1)$ be a piecewise constant function and let~$u_1$ be the solution of problem~\eqref{2.2}, \eqref{1.2} with $\eta=\eta_1$. To simplify notation, we assume that there are only two intervals of constancy for~$\eta_1(t)$ and  write
$$
\eta_1(t,x)=I_{J_1}(t)\eta_1^1(x)+I_{J_2}(t)\eta_1^2(x), 
$$
where $\eta_1^1,\eta_1^2\in G_1$ are some vectors and $J_1=[0,a]$ and $J_2=[a,T]$ with $a\in(0,T)$. We fix a small $\nu>0$ and, for $i=1,2$, choose numbers~$\alpha_j^i>0$, $j=1,\dots,k_i$, and vectors $\eta^i,\zeta^{ji}\in G$ such that~\eqref{4.2}, \eqref{4.3} hold. Let us
consider the following equation on~$J_T$:
\begin{equation} \label{4.7}
\p_t u-\mu\p_x^2u+\sum_{j=1}^{k_i}\alpha_j^i\bigl(\BB(u+\zeta^{ji}(x))-\mu\p_x^2\zeta^{ji}(x)\bigr)=h(t,x)+\eta^i(x), \quad t\in J_i.
\end{equation}
This is a Burgers-type equation, and using the same arguments as in the proof of Theorem~\ref{t3.2}, it can be proved that
problem~\eqref{4.7}, \eqref{1.2} has a unique solution $\tilde
u(t,x)$ satisfying~\eqref{3.9}. Moreover, in view of the regularity of the data and an analogue of Proposition~\ref{p3.3} for Eq.~\eqref{4.7}, we have 
\begin{equation} \label{4.50}
\tilde u\in C(J_T,H_{\rm ul}^k)\quad \mbox{for any $k\ge0$}. 
\end{equation}
On the other hand, we can rewrite~\eqref{4.7} in the form
\begin{equation} \label{4.8}
\p_t u-\mu \p_x^2u+u\p_xu=h(t,x)+\eta_1^i(x)-c_\nu^i(t,x),\quad t\in J_i, 
\end{equation}
where $c_\nu^i(t,x)$ is defined for $t\in J_i$ by the function under sign of norm on the left-hand side of~\eqref{4.3} in which $\eta_1=\eta_1^i$, $\eta=\eta^i$, $\alpha_j=\alpha_j^i$, $\zeta^j=\zeta^{ji}$, and $u=\tilde u(t,x)$.
Since the resolving operator for~\eqref{4.8} is Lipschitz continuous on bounded subsets, there is a constant $C>0$ depending only on the $L^\infty$~norms of~$\eta_1^i$ such that (see Remark~\ref{r3.1})
\begin{equation} \label{4.9}
\|u_1(T)-\tilde u(T)\|_{H_{\rm ul}^1}
\le C\bigl(\|c_\nu^1\|_{L^2(J_1,L^\infty)}+\|c_\delta^2\|_{L^2(J_2,L^\infty)}\bigr)
\le C\sqrt{2T}\,\nu.
\end{equation}
On the other hand, let us define $\eta\in L^\infty(J_T,G)$ by $\eta(t)=\eta^i$ for $t\in J_i$.  We shall show in the next steps that  there is a sequence $\{\zeta_m\}\subset L^\infty(J_T,G)$ such that
\begin{equation} \label{4.10}
\|u^m(T)-\tilde u(T)\|_{H_{\rm ul}^1}\to0\quad\mbox{as $m\to\infty$},
\end{equation}
where $u^m(t,x)$ denotes the solution of problem~\eqref{2.3}, \eqref{1.2} in which $\zeta=\zeta_m$. Combining inequalities~\eqref{4.9} and~\eqref{4.10} with $\nu\ll1$ and~$m\gg1$, we obtain the required estimate~\eqref{4.1} for $u=u^m$.

\medskip
{\sl Step~3: Fast oscillating controls}.
Following a classical idea in the control theory, we define functions
$\zeta_m\in L^\infty(J_T,G)$ by the relation
$$
\zeta_m(t)=\left\{
\begin{array}{cl}
\zeta^{(1)}(mt/a)&\quad\mbox{for $t\in J_1$},\\[4pt]
\zeta^{(2)}(m(t-a)/(T-a))&\quad\mbox{for $t\in J_2$},
\end{array}
\right.
$$
where $\zeta^{(i)}(t)$ is a $1$-periodic $G$-valued function such that
$$
\zeta^{(i)}(t)=\zeta^{ji}\quad\mbox{for $0\le t-(\alpha_1^i+\cdots+\alpha_{j-1}^i)< \alpha_j^i$, $j=1,\dots,k_i$}. 
$$
Let us rewrite~\eqref{4.7} in the form
$$
\p_t u-\mu \p_x^2(u+\zeta_m(t,x))+\BB(u+\zeta_m(t,x))=h(t,x)+\eta(t,x)+f_m(t,x),
$$
where we set $f_m=f_{m1}+f_{m2}$,
\begin{align}
f_{m1}(t)&=-\mu \p_x^2\zeta_m+\mu\sum_{j=1}^{k_i}\alpha_j^i\p_x^2\zeta^{ji},\label{4.11}\\
f_{m2}(t)&=\BB(\tilde u+\zeta_m)-\sum_{j=1}^{k_i}\alpha_j^i \BB(\tilde u+\zeta^{ji})\label{4.12}
\end{align}
for $t\in J_i$. We now define an operator $\KK:L^2(J_T,L^\infty)\to L^\infty(J_T\times\R)\cap C_*(J_T,L_{\rm ul}^2)$ by the relation  
$$
(\KK f)(t,x)=\int_0^tK_{t-s}*f(s)\,ds,
$$
where the kernel~$K_t$ was introduced in~\eqref{3.00}. Setting $v_m=\tilde u-\KK f_m$, we see that the function~$v_m(t,x)$ satisfies the equation
\begin{equation} %\label{51}
\p_t v-\mu \p_x^2(v+\zeta_m)+\BB(v+\zeta_m+\KK f_m)=h+\eta.
\end{equation}
Suppose we have shown that
\begin{equation} \label{52}
\|\KK f_m(T)\|_{H_{\rm ul}^1}+\|\KK f_m\|_{L^4(J_T,W^{2,\infty})}\to0
\quad\mbox{as $m\to\infty$}.
\end{equation}
Then, by~\eqref{3.33}, we have
$$
\|u^m(T)-\tilde u(T)\|_{H_{\rm ul}^1}
\le \|u^m(T)-v_m(T)\|_{H_{\rm ul}^1}+\|\KK f_m(T)\|_{H_{\rm ul}^1}\to0
\quad\mbox{as $m\to\infty$}.
$$
Thus, it remains to prove~\eqref{52}.

\medskip
{\sl Step~4: Proof of~\eqref{52}}. We first note that $\{f_m\}$ is a bounded sequence in $L^\infty(J_T,H_{\rm ul}^k)$ for any $k\ge0$. Integrating by parts, it follows that 
\begin{equation} \label{4.13}
\KK f_m=F_m+\mu\,\KK (\p_x^2F_m), 
\end{equation}
where we set 
$$
F_m(t)=\int_0^tf_m(s)\,ds. 
$$
In view of Proposition~\ref{p3.1}, the operator~$\KK$ is continuous from~$L^1(J_T,H_{\rm ul}^k)$ to $C(J_T,H_{\rm ul}^k)$ for any integer $k\ge0$. Therefore~\eqref{52} will follow if we show that
$$
\|F_m\|_{C(J_T,H_{\rm ul}^k)}\to0\quad\mbox{as $m\to\infty$}.
$$
This convergence is a straightforward consequence of relations~\eqref{4.11} and~\eqref{4.12}; e.g., see~\cite[Section~3.3]{shirikyan-cmp2006}. 
The proof of Proposition~\ref{p2.5} is complete.

\subsection{Saturation}
\label{s4.3}
We wish to prove~\eqref{2.7}. To this end, we shall need the following lemma describing explicitly some subspaces that are certainly included in~$E_k$.  Without loss of generality, we assume that $\lambda_1>\lambda_2$. 

\begin{lemma} \label{l2.7}
Let us set~$\Lambda_k=\{n_1\lambda_1+n_2\lambda_2\ge0: n_1,n_2\in\Z, |n_1|+|n_2|\le k\}$. Then $E_{\Lambda_k}\subset E_k$ for any integer $k\ge1$.
\end{lemma}

\begin{proof}
The proof is by induction on~$k$. We confine ourselves to carrying out the induction step, since the base of induction can be checked by a similar argument. 

Let us fix any integer $k\ge2$ and assume that $E_{\Lambda_k}\subset E_k$. We need to show that that the functions $\sin(\lambda x)$ and $\cos(\lambda x)$ belong to~$E_{k+1}$ for $\lambda=n_1\lambda_1+n_2\lambda_2\in\Lambda_{k+1}$. We shall only consider the case when the coefficients~$n_1$ and~$n_2$ are non-negative, since the other situations can be treated by similar arguments. Assume first   $n_1\ge2$ and $n_1+n_2\le k+1$. Then $\lambda'=\lambda-\lambda_1$ and $\lambda''=\lambda-2\lambda_1$ belong to~$\Lambda_k$, and we have
\begin{align}
\sin(\lambda x)&=\tfrac{\lambda''}{\lambda}\sin(\lambda''x)
+\tfrac{2}{\lambda}\bigl(\sin(\lambda_1x)\,\p_x\sin(\lambda'x)
+\sin(\lambda'x)\,\p_x\sin(\lambda_1 x)\bigr),\label{4.31}
\\
\cos(\lambda x)&=-\tfrac{\lambda}{\lambda''}\cos(\lambda''x)
+\tfrac{2}{\lambda''}\bigl(\cos(\lambda_1x)\,\p_x\sin(\lambda'x)
+\sin(\lambda'x)\,\p_x\cos(\lambda_1x)\bigr),\label{4.32}
\end{align}
whence we conclude that the functions on the left-hand side of these relations belong to~$E_{k+1}$. If $\lambda=\lambda_1+k\lambda_2\in\Lambda_{k+1}$, then setting $\lambda'=\lambda-\lambda_2$ and $\lambda''=\lambda-2\lambda_2$, we see that relations~\eqref{4.31} and~\eqref{4.32} with~$\lambda_1$ replaced by~$\lambda_2$ remain valid, and we can conclude again that $\sin(\lambda x),\cos(\lambda x)\in E_{k+1}$. Finally, the same proof applies also in the case $\lambda=(k+1)\lambda_2\in \Lambda_{k+1}$.
\end{proof}

Lemma~\ref{l2.7} shows that the union of~$E_k$ (which is a vector space) contains the trigonometric functions whose frequencies belong to the set $\Lambda_\infty:=\cup_k\Lambda_k$. It is straightforward to check that~$\Lambda_\infty$ is dense in~$\R_+$. 

\subsection{Large control space}
\label{s4.4}
Let us prove that~\eqref{2.2} is $E_{\Lambda_k}$-controllable (and, hence, $E_k$-controllable) for a sufficiently large~$k$. Indeed, let us set
\begin{equation} \label{4.33}
u(t,x)=T^{-1}\bigl(t\hat u(x)+(T-t)u_0(x)\bigr), \quad (t,x)\in J_T\times\R. 
\end{equation}
This is an infinity smooth function in~$(t,x)$ all of whose derivatives are bounded. We now define
$$
\eta(t,x)=\p_tu-\mu\p_x^2u+u\p_xu-h
$$
and note that $\eta\in L^2(J_T,H_{\rm ul}^s)$ for any $s\ge0$ and that the solution of problem~\eqref{2.2}, \eqref{1.2} is given by~\eqref{4.33} and coincides with~$\hat u$ for $t=T$. We have thus a control that steers a solution starting from~$u_0$ to~$\hat u$. To prove the required property, we approximate~$\eta$, in local topologies, by an $E_{\Lambda_k}$-valued function and use the continuity of the resolving operator to show that the corresponding solutions are close. 

More precisely, let $\chi\in C^\infty(\R)$ be such that $0\le\chi\le1$, $\sup_\R|\chi'|\le 2$, $\chi(x)=0$ for $|x|\ge2$, and $\chi(x)=1$ for $|x|\le1$. Then the sequence $\eta_n(t,x)=\chi(x/n)\eta(t,x)$ possesses the following properties:
\begin{gather}
\eta_n(t,x)=0\quad\mbox{for $|x|\ge2n$ and any $n\ge1$},\label{4.34}\\
\|\eta_n\|_{L^2(J_T,H_{\rm ul}^1)}\le 3\|\eta\|_{L^2(J_T,H_{\rm ul}^1)}
\quad\mbox{for all $n\ge1$},\label{4.35}\\
\|\eta_n-\eta\|_{L^2(J_T\times I_\rho)}\to0\quad\mbox{as $n\to\infty$ for any $\rho>0$}, \label{4.36}
\end{gather}
where $I_\rho=[-\rho,\rho]$. Given a frequency $\omega>0$ and an integer $N\ge1$, we denote by ${\mathsf P}_{\omega,N}:L^2(I_{\pi/\omega})\to L^\infty(\R)$ a linear projection that takes a function~$g$ to its truncated Fourier series
$$
({\mathsf P}_{\omega,N}g)(x)=\sum_{|j|\le N}g_je^{\omega i j x}, \quad
g_j=\frac{\omega}{2\pi}\int_{I_{\pi/\omega}}g(y)e^{-\omega i j y}dy. 
$$
The function ${\mathsf P}_{\omega,N}g$ is $2\pi/\omega$-periodic, and it follows from~\eqref{4.34} and~\eqref{4.35} that
\begin{gather}
\|{\mathsf P}_{\omega,N}\eta_n\|_{L^1(J_T,L^\infty)}
\le C_1\|{\mathsf P}_{\omega,N}\eta_n\|_{L^2(J_T,H_{\rm ul}^1)}
\le C_2\quad\mbox{for all $N,n\ge1$},\label{4.37}\\
\|{\mathsf P}_{\omega,N}\eta_n-\eta_n\|_{L^2(J_T\times I_\rho)}\to0
\quad\mbox{as $N\to\infty$ for any $n\ge1$}. 
\label{4.38}
\end{gather}
Note that if $\omega\in\Lambda_\infty$, then for any $N\ge1$ there is $k\ge1$ such that the image of~${\mathsf P}_{\omega,N}$ is contained in~$E_{\Lambda_k}$. 

Let us denote by~$u_{n,N}(t,x)$ the solution of problem~\eqref{2.2}, \eqref{1.2} with $\eta={\mathsf P}_{\omega,N}\eta_n$. In view of inequality~\eqref{3.18} with $\delta=\e/2$ and $R=\max\{\|u_0\|_{L^\infty},\|\eta\|_{L^1(J_T,L^\infty)},C_2\}$, we have 
\begin{multline} \label{4.39}
\|u_{n,N}(T)-\hat u\|_{L^2(I_r)}=\|u_{n,N}(T)-u(T)\|_{L^2(I_r)}
\le\tfrac\e2+C\,\|{\mathsf P}_{\omega,N}\eta_n-\eta\|_{L^1(J_T,L^2(I_\rho))}
\\
\le \tfrac\e2+C\sqrt{T}\,\Bigl(\|{\mathsf P}_{\omega,N}\eta_n-\eta_n\|_{L^1(J_T,L^2(I_\rho))}+\|\eta_n-\eta\|_{L^1(J_T,L^2(I_\rho))}\Bigr).
\end{multline}
We now choose $n\ge1$ such that $C\sqrt{T}\,\|\eta_n-\eta\|_{L^1(J_T,L^2(I_\rho))}<\frac\e4$; see~\eqref{4.36}. We next find $\omega\in\Lambda_\infty$ so that $\frac\pi\omega>\max(2n,\rho)$ (this is possible since~$\Lambda_\infty$ is dense in~$\R_+$) 
and choose~$N\ge1$ such that $C\sqrt{T}\|{\mathsf P}_{\omega,N}\eta_n-\eta_n\|_{L^1(J_T,L^2(I_\rho))}<\frac\e4$. Substituting these estimates into~\eqref{4.39}, we obtain 
$$
\|u_{n,N}(T)-\hat u\|_{L^2(I_r)}<\e,
$$
which is the second inequality in~\eqref{2.1} with $s=0$. It remains to note that, in view of~\eqref{3.10}, \eqref{4.35}, and~\eqref{4.37}, the first inequality in~\eqref{2.1} is also satisfied.

\subsection{Reduction to the case $s=0$}
\label{s4.5}
We now prove that if inequalities~\eqref{2.1} hold for $s=0$ and arbitrary~$T$, $r$,  and~$\e$, then they remain valid for any $s\ge1$. Indeed, we fix an integer $s\ge1$, positive numbers~$r$ and~$\e$, and functions~$u_0,\hat u\in C_b^\infty$. Let us define~$\eta$ by zero on the half-line $[T,+\infty)$ and denote by~$\hat u(t)$ the solution of~\eqref{1.1}, \eqref{1.3} issued from~$\hat u$ at $t=T$. Using interpolation, regularity of solutions (Proposition~\ref{p3.3}), and continuity of the resolving operator in local norms (Proposition~\ref{p3.4}), we can write
\begin{multline}
\|u(T+\tau)-\hat u(\tau)\|_{H^s(I_r)}^2
\le C_1\|u(T+\tau)-\hat u(\tau)\|_{L^2(I_r)}\|u(T+\tau)-\hat u(\tau)\|_{H^{2s}(I_r)}
\\
\le C_2\tau^{-2s}\bigl(\delta+C\,\|u(T)-\hat u\|_{L^2(I_\rho)}\bigr)\,
Q_{2s}\bigl(\|u(T)\|_{L^\infty}+K\bigr),
\label{4.51}
\end{multline}
where $C_i$ are some constants depending on~$R$ and~$s$,  the quantities~$C$ and~$Q_{2s}$ are those entering~\eqref{3.18} and~\eqref{3.16}, respectively, and $K=\|\hat u\|_{L^\infty}+\|h\|_{L^1(J_T,H_{\rm ul}^{2s})}$. Furthermore, in view of Proposition~\ref{p3.3}, we have
\begin{equation*}
\|\hat u(\tau)-\hat u\|_{H_{\rm ul}^s}\to0\quad\mbox{as $\tau\to0^+$}. 
\end{equation*}
Let~$\tau>0$ be so small that the left-hand side of this relation is smaller than~$\e^2/6$. We next choose~$\delta>0$ such that 
$$
C_2\tau^{-2s} Q_{2s}(K_0+K)\delta <\e^2/6,
$$
where $K_0$ is defined in~\eqref{2.1} (and is independent of~$r$ and~$\e$). Finally, we construct $\eta\in C^\infty(J_T,E_\Lambda)$ for which inequalities~\eqref{2.1} hold with $r=\rho$ and $\e=\delta/C$. Comparing the above estimates with~\eqref{4.51}, we obtain
$$
\|u(T+\tau)-\hat u\|_{H_{\rm ul}^s(I_r)}
:=\sup_{I\subset I_r}\|u(T+\tau)-\hat u\|_{H^s(I)}<\e,
$$
where the supremum is taken oven all intervals $I\subset I_r$ of length~$\le1$. Furthermore, in view of~\eqref{3.16}, we have
$$
\|u(T+\tau)\|_{H_{\rm ul}^s}\le \tau^{-s}Q_{s}\bigl(K_0+\|h\|_{L^1(J_T,H_{\rm ul}^s)}\bigl)=:K_s. 
$$
We have thus established inequalities~\eqref{2.1} with~$T$ and $\|\cdot\|_{H^s(I_r)}$ replaced by~$T+\tau$ and~$\|\cdot\|_{H_{\rm ul}^s(I_r)}$, respectively. Since~$T$ is arbitrary and the positive numbers~$\tau$ and~$\e$ can be chosen arbitrarily small, we conclude that inequalities~\eqref{2.1} are true for any integer~$s\ge0$ and any numbers $T,r,\e>0$. This completes the proof of Theorem~\ref{t2.1}.

\def\cprime{$'$} \def\cprime{$'$}
  \def\polhk#1{\setbox0=\hbox{#1}{\ooalign{\hidewidth
  \lower1.5ex\hbox{`}\hidewidth\crcr\unhbox0}}}
  \def\polhk#1{\setbox0=\hbox{#1}{\ooalign{\hidewidth
  \lower1.5ex\hbox{`}\hidewidth\crcr\unhbox0}}}
  \def\polhk#1{\setbox0=\hbox{#1}{\ooalign{\hidewidth
  \lower1.5ex\hbox{`}\hidewidth\crcr\unhbox0}}} \def\cprime{$'$}
  \def\polhk#1{\setbox0=\hbox{#1}{\ooalign{\hidewidth
  \lower1.5ex\hbox{`}\hidewidth\crcr\unhbox0}}} \def\cprime{$'$}
  \def\cprime{$'$} \def\cprime{$'$} \def\cprime{$'$}
\providecommand{\bysame}{\leavevmode\hbox to3em{\hrulefill}\thinspace}
\providecommand{\MR}{\relax\ifhmode\unskip\space\fi MR }
\providecommand{\MRhref}[2]{%
  \href{http://www.ams.org/mathscinet-getitem?mr=#1}{#2}
}
\providecommand{\href}[2]{#2}

%\input{referenc}
%\bibliography{references}
%\bibliographystyle{plain}
%\bibliographystyle{amsplain}
\end{document}